\documentclass[12pt]{amsart}
\usepackage{url}

\usepackage[top=1.25in, bottom=1.25in, left=1.05in, right=1.05in]{geometry}

\usepackage{amsthm,amsmath,amssymb,graphicx}
\usepackage[colorlinks=true,citecolor=black,linkcolor=black,urlcolor=blue]{hyperref}

\theoremstyle{plain}
\newtheorem{definition}{Definition}[section]
\newtheorem{remark}[definition]{Remark}

\newtheorem{lemma}[definition]{Lemma}
\newtheorem{theorem}[definition]{Theorem}
\newtheorem{corollary}[definition]{Corollary}

\numberwithin{equation}{section}
\begin{document}

\title[Hankel determinants for the generalized derangement polynomials]{\large The Hankel determinants for the generalized derangement polynomials of order $r$}

\author[G. Guettai]{GHANIA GUETTAI}

\address[G. Guettai]{University Yahia Far\`{e}s M\'{e}d\'{e}a\\ Faculty of Science\\ urban pole, 26000\\ M\'{e}d\'{e}a, Algeria}
\email{guettai78@yahoo.fr}
\author[D. Laissaoui]{DIFFALAH LAISSAOUI}
\address[D. Laissaoui]{University Yahia Far\`{e}s M\'{e}d\'{e}a\\ Faculty of Science\\ urban pole, 26000\\ M\'{e}d\'{e}a, Algeria}
\email{laissaoui.diffalah74@gmail.com}

\author[M. Rahmani]{MOURAD RAHMANI}
\address[M. Rahmani]{Faculty of Mathematics\\ USTHB \\ P.O. Box 32 El Alia 16111\\ Algiers \\ Algeria.}
\email{rahmani.mourad@gmail.com, mrahmani@usthb.dz}

\vspace{20pt}

\begin{abstract}
This paper sets out to introduce the generalized derangement polynomials of order $r $. It then proceeds to establish various identities associated with these polynomials, along with providing recurrence relations for derangement polynomials of order $ r$. Additionally, the paper offers a probabilistic approach for the generalized derangement polynomials of order $r $. Furthermore, it furnishes a clear expression for the Hankel determinants pertaining to these generalized derangement polynomials of order $r$, and subsequently infers the Hankel determinants for derangement polynomials of the same order, as well as for the count of cyclic derangements.
\end{abstract}
\maketitle

\section{Introduction}
A derangement refers to a permutation where none of the elements occupy their original positions. The derangement number, denoted as $D_{n}$, represents the count of such permutations for a set of  $n$ elements (also known as the number of fixed point-free permutations). The task of enumerating derangements was first tackled by Pierre R\'{e}mond de Montmort (1678-1719) and was referred to as "le probl\`{e}me de rencontres" in French. The first few terms of the derangement number are $D_{0}=1$, $D_{1}=0$, $D_{2}=1$, $D_{3}=2$, $D_{4}=9,\ldots$. \\A simple explicit expression for $D_{n}$\ is (see \cite{Bon,Car})
\begin{equation}
	D_{n}=n!\sum_{k=0}^{n}\frac{\left(  -1\right)  ^{k}}{k!}.\label{1}
\end{equation}
From (\ref{1}), we note that the exponential generating function of derangement numbers is given by
\begin{equation}
	\frac{e^{-z}}{1-z}=\sum_{n\geq0}D_{n}\frac{z^{n}}{n!}.\label{2}
\end{equation}

A fixed point-free permutations on $n+r$ letters will be called fixed point-free $r$-permutation if in its cycle decomposition the first $r$ letters appear to be in distinct cycles. The number of fixed point-free $r$-permutations denote by $D_{n}^{(r)}$ and call $r$-derangement number $(r\in\mathbb{N}^{\ast})$. The first $r$ elements, as well as the cycles they are contained in will be called distinguished. For any $r\in\mathbb{N}^{\ast}$, the exponential generating function of the sequence of $r$-derangements numbers is given by (see \cite{Mez})
\begin{equation}
	\frac{z{^{r}}e^{-z}}{\left(1-z\right)^{r+1}}=\sum_{n\geq0}D_{n}^{(r)}\frac{z^{n}}{n!}.\label{3}
\end{equation}

The derangement numbers $d_{n}^{(r)}$ of order $r$ $(n\geq0)$, are defined by the exponential generating function (see \cite{Kim})
\begin{equation}\label{2-1}
	\frac{e^{-z}}{\left(  1-z\right)^{r}}=\sum_{n\geq0}d_{n}^{(r)}\frac{z^{n}}{n!}.
\end{equation}

Many generalizations and extensions have been published in recent years. Firstly, the $r$-derangement polynomials $D_{n}^{(r)}(x)$, are defined by the exponential generating function (see \cite{Mez})
\begin{equation}
	\frac{z{^{r}}e^{xz}}{\left(1-z\right)  ^{r+1}}=\sum_{n\geq0}D_{n}^{(r)}(x)\frac{z^{n}}{n!}.\label{4}
\end{equation}
Secondly, the derangement polynomials $d_{n}^{(r)}(x) $ of order $r$, are defined by the exponential generating function (see \cite{Kim})
\begin{equation}
	\frac{e^{xz}}{\left(  1-z\right)^{r}}=\sum_{n\geq0}d_{n}^{(r)}(x)\frac{z^{n}}{n!}.\label{5}
\end{equation}
Another generalization of the derangement number is the number of cyclic derangements $d_{n,r}$, are defined by the exponential generating function (see \cite{Assaf})
\begin{equation}
	\frac{e^{-z}}{1-rz}=\sum_{n\geq0}d_{n,r}\frac{z^{n}}{n!}.\label{6}
\end{equation}

\section{The generalized derangement polynomials of order $r$}
This section introduces a novel set of polynomials termed as the "generalized derangement polynomials of order  $r$". We will outline their characteristics and elucidate their connection with derangement polynomials of the same order. Following this, we will provide the recurrence relations governing these derangement polynomials of degree  $r$.
\begin{definition}
	\label{def1}We define the generalized derangement polynomials $ \mathbb{D}_{n}^{(r)}\left(  x\right) $ of order $r$ by
	\begin{equation}
		\frac{e^{z}}{\left(  1-xz\right)  ^{r}}=\sum_{n\geq0}\mathbb{D}_{n}
		^{(r)}\left(  x\right)  \frac{z^{n}}{n!}.\label{D1}
	\end{equation}
\end{definition}
\begin{theorem}
	For $n\geq0$, we have the explicit formula of the generalized derangement polynomials of order $r$
	\begin{equation}
		\mathbb{D}_{n}^{(r)}\left(  x\right)  =\sum_{k=0}^{n}\binom{n}{k}r^{\overline{k}}x^{k}.
	\end{equation}
	where $ r^{\overline{k}}$ is the rising factorial, defined by $r^{\overline{k}}=r(r+1)...(r+k-1) $, with $r^{\overline{0}}=1 $.	
\end{theorem}
\begin{proof}
	By using the exponential generating function (\ref{D1}), we obtain
	\begin{align*}
		\sum_{n\geq0}\mathbb{D}_{n}^{(r)}\left(  x\right)  \frac{z^{n}}{n!}  &
		=\frac{e^{z}}{\left(  1-xz\right)  ^{r}}\\
		&  =\sum_{k\geq0}\frac{r^{\overline{k}}(xz)^{k}}{k!}
		\sum_{n\geq0}\frac{z^{n}}{n!}\\
		&  =\sum_{k\geq0}r^{\overline{k}}x^{k}\sum_{n\geq0}
		\frac{z^{k}z^{n}}{k!n!}\\
		&  =\sum_{k\geq0}r^{\overline{k}}x^{k}\sum_{n\geq0}
		\binom{n+k}{k}\frac{z^{n+k}}{(n+k)!}\\
		&  =\sum_{k\geq0}r^{\overline{k}}x^{k}\sum_{n\geq k}\binom
		{n}{k}\frac{z^{n}}{n!}\\
		&  =\sum_{n\geq0}\left(  \sum_{k=0}^{n}\binom{n}{k}r^{\overline{k}}x^{k}\right)  \frac{z^{n}}{n!}.
	\end{align*}
	Equating the coefficients of $\frac{z^{n}}{n!}$ in both sides of the above expression, we get the desired result.
\end{proof}
The recurrence relation of the generalized derangement polynomials $\mathbb{D}_{n}^{(r)}\left(x\right) $ of order $r$ is given by
\begin{theorem}
	For $n\geq0$, we have
	\begin{equation}\label{T3}
		\mathbb{D}_{n+1}^{(r)}\left(  x\right)  =\mathbb{D}_{n}^{(r)}\left(  x\right)
		+rx\mathbb{D}_{n}^{(r+1)}\left(  x\right).
	\end{equation}	
\end{theorem}
\begin{proof}
	By derivation the exponential generating function (\ref{D1}) with respect to $z$, we obtain
	\begin{align*}
		\frac{d}{dz}\left(  \sum_{n\geq0}\mathbb{D}_{n}^{(r)}\left(  x\right)
		\frac{z^{n}}{n!}\right)   &  =\frac{d}{dz}\left(  \frac{e^{z}}{\left(
			1-xz\right)  ^{r}}\right) \\
		&  =\frac{e^{z}}{\left(  1-xz\right)  ^{r}}+\frac{rxe^{z}}{\left(
			1-xz\right)  ^{r+1}}\\
		&  =\sum_{n\geq0}\mathbb{D}_{n}^{(r)}\left(  x\right)  \frac{z^{n}}{n!}%
		+rx\sum_{n\geq0}\mathbb{D}_{n}^{(r+1)}\left(  x\right)  \frac{z^{n}}{n!}\\
		&  =\sum_{n\geq0}\left( \mathbb{D}_{n}^{(r)}\left(  x\right)
		+rx\mathbb{D}_{n}^{(r+1)}\left(  x\right)\right)   \frac{z^{n}}{n!}.
	\end{align*}
	On the other hand, we have
	\[
	\frac{d}{dz}\left(  \sum_{n\geq0}\mathbb{D}_{n}^{(r)}\left(  x\right)
	\frac{z^{n}}{n!}\right)  =\sum_{n\geq0}\mathbb{D}_{n+1}^{(r)}\left(  x\right)
	\frac{z^{n}}{n!}.
	\]
	Equating the coefficients of $\frac{z^{n}}{n!}$ in both sides of the above expression, we get the desired result.
\end{proof}
\begin{theorem}
	For $n\geq0$, we have
	\begin{equation}\label{T4}
		\mathbb{D}_{n+1}^{(r)}\left(  x\right)  =\mathbb{D}_{n}^{(r)}\left(  x\right)
		+rx\sum_{k=0}^{n}\binom{n}{k}\mathbb{D}_{k}^{(r)}\left(  x\right)
		x^{n-k}\left(  n-k\right)!.
	\end{equation}
\end{theorem}
\begin{proof}
	Using the derivative of the generating function (\ref{D1}) with respect to $z$, we obtain
	\begin{align*}
		\frac{d}{dz}\left(  \sum_{n\geq0}\mathbb{D}_{n}^{(r)}\left(  x\right)
		\frac{z^{n}}{n!}\right)   &  =\frac{d}{dz}\left(  \frac{e^{z}}{\left(
			1-xz\right)  ^{r}}\right) \\
		&  =\frac{e^{z}}{\left(  1-xz\right)  ^{r}}+\frac{rxe^{z}}{\left(
			1-xz\right)  ^{r+1}}\\
		&  =\sum_{n\geq0}\mathbb{D}_{n}^{(r)}\left(  x\right)  \frac{z^{n}}{n!}
		+rx\sum_{n\geq0}\mathbb{D}_{n}^{(r)}\left(  x\right)  \frac{z^{n}}{n!}
		\sum_{n\geq0}n!x^{n}\frac{z^{n}}{n!}.
	\end{align*}
	Thus,
	\[
	\sum_{n\geq0}\mathbb{D}_{n+1}^{(r)}\left(  x\right)  \frac{z^{n}}{n!}
	=\sum_{n\geq0}\left(\mathbb{D}_{n}^{(r)}\left(  x\right)+rx  \sum_{k=0}^{n}\binom{n}{k}\mathbb{D}_{k}^{(r)}\left(
	x\right)  x^{n-k}\left(  n-k\right)  !\right)  \frac{z^{n}}{n!}.
	\]
	Equating the coefficients of $\frac{z^{n}}{n!}$ in both sides of the above expression, we get the desired result.	
\end{proof}
Here we present some properties of the derangement polynomials $d_{n}^{(r)}(x) $ of order $r$ given by the formula (\ref{5}). We start with the explicit form and then provide two recurrence relations.
\begin{theorem}
	For $n\geq0$, we have the explicit formula of the derangement polynomials $d_{n}^{(r)}(x) $ of order $r$
	\begin{equation}
		{d}_{n}^{(r)}\left(  x\right)  =\sum_{k=0}^{n}\binom{n}{k}r^{\overline{k}}x^{n-k}.
	\end{equation}
	where $ r^{\overline{k}}$ is the rising factorial.	
\end{theorem}
\begin{proof}
	By using the exponential generating function (\ref{5}), we obtain
	\begin{align*}
		\sum_{n\geq0}{d}_{n}^{(r)}\left(  x\right)  \frac{z^{n}}{n!}  &
		=\frac{e^{xz}}{\left(  1-z\right)  ^{r}}\\
		&  =\sum_{k\geq0}\frac{r^{\overline{k}}z^{k}}{k!}
		\sum_{n\geq0}\frac{(xz)^{n}}{n!}\\
		&  =\sum_{k\geq0}r^{\overline{k}}\sum_{n\geq0}
		\frac{x^{n}z^{n+k}}{k!n!}\\
		&  =\sum_{k\geq0}r^{\overline{k}}\sum_{n\geq0}
		\binom{n+k}{k}x^{n}\frac{z^{n+k}}{(n+k)!}\\
		&  =\sum_{k\geq0}r^{\overline{k}}\sum_{n\geq k}\binom
		{n}{k}x^{n-k}\frac{z^{n}}{n!}\\
		&  =\sum_{n\geq0}\left(  \sum_{k=0}^{n}\binom{n}{k}r^{\overline{k}}x^{n-k}\right)  \frac{z^{n}}{n!}.
	\end{align*}
	Equating the coefficients of $\frac{z^{n}}{n!}$ in both sides of the above expression, we get the desired result.
\end{proof}
Now, it is easy to show the relationship between the polynomials $\mathbb{D}_{n}^{(r)}\left(  x\right)  $\ and the derangement polynomials $d_{n}^{\left(  r\right)  }\left(  x\right)$ of order $r$. For $n\geq0$ and $ x \neq 0$, we have
\begin{equation}\label{dn}
	x^{n}\mathbb{D}_{n}^{(r)}\left(  \frac{1}{x}\right)  =d_{n}^{\left(  r\right)}\left(  x\right)  .
\end{equation}
Then, we obtain the following corollaries
\begin{corollary}
	For $n\geq0$ and $ x \neq 0$, we have
	\begin{equation}\label{C5}
		d_{n+1}^{(r)}\left(  x\right)  =xd_{n}^{(r)}\left(  x\right)  +rd_{n}^{(r+1)}\left(  x\right)  \text{.}
	\end{equation}
\end{corollary}
\begin{proof}
	Using the formula (\ref{T3}) and replacing $x$ by  $\frac{1}{x}$, after that multiplying the both sides of the new expression by $x^{n+1}$. Finally, by using the formula (\ref{dn}), we obtain the desired formula (\ref{C5}).
\end{proof}	
\begin{corollary}
	For $n\geq0$ and $ x \neq 0$, we have
	\begin{equation}
		d_{n+1}^{(r)}\left(  x\right)  =xd_{n}^{(r)}\left(  x\right)  +r\sum_{k=0}
		^{n}\binom{n}{k}d_{k}^{(r)}\left(  x\right)  \left(  n-k\right)!.
	\end{equation}
\end{corollary}
\section{Probabilistic approach}
In this section, we will present a probabilistic representation of the generalized derangement polynomials of order $ r$. Let $f_{X}(x)$ be the probability density function of the continuous random variable $X$, and let $h(x)$ be a real valued function. Then the expectation
of $h(X)$, is defined by (see \cite{Ross})
\[
E(h(X))={\displaystyle\int\limits_{-\infty}^{+\infty}}
h(x)f_{X}(x)dx.
\]
A continuous random variable $X$, is said to be the gamma random variable with parameters $\alpha,\beta$ ; for some $\alpha>0$ and $\beta>0$, if its density function is given by
\[
f_{X}(x)=\genfrac{\{}{.}{0pt}{}{\dfrac{\beta^{\alpha}x^{\alpha-1}e^{-\beta x}}{\Gamma(\alpha)}\text{ \ \ if }x\geq0,}{0\text{ \ }else}
\]
where $\Gamma(\alpha)$ is the gamma function, is denoted $X\sim\Gamma(\alpha,\beta)$.
\begin{lemma}
	Let $Y\sim\Gamma(1,1).$ Then, for all $x<1$, we have
	\[
	e^{z}\phi_{Y}(z)=\sum_{n\geq0}D_{n}\frac{z^{n}}{n!},
	\]
	where $\phi_{Y}$ is the moment generating function of random variable $Y$.
\end{lemma}
\begin{corollary}
	Let $X_{1},X_{2},\ldots,X_{r}$ be independent and identically distributed gamma random variables with shape $\alpha=1$ and rate $\beta=1$, we assume
	$Y_{r}={\displaystyle\sum\limits_{l=1}^{r}}	X_{l}$, then
	\[
	e^{z}\phi_{Y_{r}}(xz)=\sum_{n\geq0}\mathbb{D}_{n}^{(r)}\left(  x\right)  \frac{z^{n}}{n!},
	\]
	where $\phi_{Y_{r}}$ is the moment generating function of random variable
	$Y_{r}.$
\end{corollary}
\begin{corollary}
	Let $X_{1},X_{2},\ldots,X_{r}$ be independent and identically distributed
	random variables with $X\sim\Gamma(1,1)$, we assume $Y_{r}=
	{\displaystyle\sum\limits_{l=1}^{r}}
	X_{l}$, then
	\begin{align*}
		\mathbb{D}_{n}^{(r)}\left(  x\right)   &  =\sum_{k=0}^{n}\binom{n}{k}
		x^{k}E(Y_{r}^{k})
	\end{align*}
	and
	\[
	d_{n}^{\left(  r\right)  }\left(  x\right)  =\sum_{k=0}^{n}\binom{n}{k}
	x^{n-k}E(Y_{r}^{k}).
	\]
\end{corollary}
\begin{corollary}
	Let $X_{1},X_{2},\ldots,X_{r}$ be independent and identically distributed	random variables with $X\sim\Gamma(1,1)$, then
	\begin{align*}
		E\left(	X_{1}+X_{2}+\cdots+X_{r}\right)^{k}   &  =r^{\overline{k}}.
	\end{align*}	
\end{corollary}
\section{On the Hankel determinant of $r-$derangement polynomials}
The Hankel matrix of order $n+1$ of a sequence $a_{0},a_{1},\ldots,$ is the $n+1$ by $n+1$ matrix whose $(i,j)$ entry is $a_{i+j},$ where the indices range between $0$ and $n$. The Hankel determinant of order $n+1$ is the determinant of the corresponding Hankel matrix, that is,
\[
\det(a_{i+j})_{0\leq i,j\leq n}=\det\left(
\begin{array}
	[c]{cccc}
	a_{0} & a_{1} & \cdots & a_{n}\\
	a_{1} & a_{2} & \cdots & a_{n+1}\\
	\vdots & \vdots & \ddots & \vdots\\
	a_{n} & a_{n+1} & \cdots & a_{2n}
\end{array}
\right).
\]

In this section,  employing a methodology akin to that in \cite{Rad}, we construct a comprehensive formulation for the Hankel determinant of the generalized derangement polynomials in terms of order $r$. Subsequently, we derive the Hankel determinant for the derangement polynomials of order $r$, as well as for the count of cyclic derangements.
\begin{lemma}\label{L11}
	\label{deriv n D}The following formula holds true
	\[
	\mathbb{D}_{n}^{(r)}\left(  \frac{1}{1-z}\right)  =\left(  1-z\right)
	^{r}e^{-z}\frac{d^{n}}{dz^{n}}\left(  \frac{e^{z}}{\left(  1-z\right)^{r}
	}\right).
	\]	
\end{lemma}
\begin{proof}
	According to Leibniz's formula, we have the $n$-th derivative of the generating function of $\mathbb{D}_{n}^{(r)}\left(  1\right)  $
	\begin{align*}
		\frac{d^{n}}{dz^{n}}\left(  \frac{e^{z}}{\left(  1-z\right)  ^{r}}\right)   &
		=\sum_{k=0}^{n}\binom{n}{k}\frac{d^{n-k}}{dz^{n-k}}\left(  e^{z}\right)
		\frac{d^{k}}{dz^{k}}\left(  \frac{1}{\left(  1-z\right)  ^{r}}\right) \\
		&  =e^{z}\sum_{k=0}^{n}\binom{n}{k} r^{\overline{k}}\frac
		{1}{\left(  1-z\right)  ^{r+k}}\\
		&  =\frac{e^{z}}{\left(  1-z\right)  ^{r}}\sum_{k=0}^{n}\binom{n}{k}
		r^{\overline{k}}\frac{1}{\left(  1-z\right)  ^{k}}\\
		&  =\frac{e^{z}}{\left(  1-z\right)  ^{r}}\mathbb{D}_{n}^{(r)}\left(  \frac
		{1}{1-z}\right).
	\end{align*}
\end{proof}
\begin{theorem}
	For $n\geq1$ and $ z \neq 1$, we have
	\begin{align}
		\label{T12}
		\det\left(  \frac{\partial^{i+j-2}}{\partial z^{i+j-2}}\left(  \frac{e^{z}
		}{\left(  1-z\right)  ^{r}}\right)  \right)  _{1\leq i,j\leq n}=\left(
		\frac{e^{z}}{\left(  1-z\right)  ^{r}}\right)  ^{n}\frac{r^{\overline{n-1}}
			{\displaystyle\prod\limits_{k=1}^{n-1}}
			r^{\overline{k-1}}k!}{\left(  z-1\right)  ^{\left(
				n-1\right)  n}}.
	\end{align}
\end{theorem}
\begin{proof}
	By induction, using Sylvestre's formula, for $x=1$, we get the numbers
	$\mathbb{D}_{n}^{(r)}\left(  1\right)  =\mathbb{D}_{n}^{(r)}$, according to lemma \ref{L11}, the generating function of $\mathbb{D}_{n}^{(r)}$ is
	\[
	G\left(  z\right)  =\sum_{n\geq0}\mathbb{D}_{n}^{(r)}\frac{z^{n}}{n!}
	=\frac{e^{z}}{\left(  1-z\right)  ^{r}}.
	\]
	We denote that the Hankel determinant  $$\left(\det\left(  \frac{\partial^{i+j-2}}{\partial z^{i+j-2}}G\left(  z\right)  \right)_{i,j=1,2,\ldots,n}\right)$$ of the $i+j-2$-th derivative of the generating function $G\left(  z\right) $ by $ \Delta_{n}\left(  G\left(  z\right)  \right)$.\\
	For $n=1$, we have
	\[
	\Delta_{1}\left(  G\left(  z\right)  \right)  =\frac{e^{z}}{\left(  1-z\right)^{r}}.
	\]
	For $n=2$, we have
	\begin{align*}
		\Delta_{2}\left(  G\left(  z\right)  \right)   &  =\left\vert
		\begin{array}
			[c]{cc}
			G\left(  z\right)  & \frac{\partial}{\partial z}G\left(  z\right) \\
			\frac{\partial}{\partial z}G\left(  z\right)  & \frac{\partial^{2}}{\partial
				z^{2}}G\left(  z\right)
		\end{array}
		\right\vert \\
		&  =\left\vert
		\begin{array}
			[c]{cc}
			\frac{e^{z}}{\left(  1-z\right)  ^{r}} & \frac{e^{z}}{\left(  1-z\right)
				^{r}}\left[  1+\frac{r}{\left(  1-z\right)  }\right] \\
			\frac{e^{z}}{\left(  1-z\right)  ^{r}}\left[  1+\frac{r}{\left(  1-z\right)
			}\right]  & \frac{e^{z}}{\left(  1-z\right)  ^{r}}\left[  1+\frac{2r}{\left(
				1-z\right)  }+\frac{r\left(  r+1\right)  }{\left(  1-z\right)  ^{2}}\right]
		\end{array}
		\right\vert \\
		&  =\left(  \frac{e^{z}}{\left(  1-z\right)  ^{r}}\right)  ^{2}\frac
		{r}{\left(  z-1\right)  ^{2}}.
	\end{align*}
	Suppose that, the formula (\ref{T12}) is true for $n=1,2,\ldots,m$ and we calculate $\Delta_{m+1}\left(G\left(  z\right)  \right)$ by Sylvester's formula
	\[
	\Delta_{m+1}\left(  G\left(  z\right)  \right)  =\frac{\Delta_{2}\left(  \Delta_{m}\left(
		G\left(  z\right)  \right)  \right)  }{\Delta_{m-1}\left(  G\left(  z\right)
		\right)  }.
	\]
	By induction hypothesis, we have
	\begin{align*}
		\Delta_{m}\left(  G\left(  z\right)  \right)   &  =\left(  \frac{e^{z}}{\left(
			1-z\right)  ^{r}}\right)  ^{m}\frac{r^{\overline{m-1}}
			{\displaystyle\prod\limits_{k=1}^{m-1}}
			r^{\overline{k-1}}k!}{\left(  z-1\right)  ^{\left(
				m-1\right)  m}}\\
		&  =\left(  r^{\overline{m-1}}
		{\displaystyle\prod\limits_{k=1}^{m-1}}
		r^{\overline{k-1}}k!\right)  \frac{e^{mz}}{\left(
			1-z\right)  ^{mr+\left(  m-1\right)  m}}.
	\end{align*}
	To calculate $\Delta_{2}\left(  \Delta_{m}\left(
	G\left(  z\right)  \right)  \right)$, we need to calculate the $1$st and the $2$nd derivative of $\Delta_{m}\left(
	G\left(  z\right)  \right)$, for that we have the $1$st derivative
	\begin{multline*}
		\frac{\partial}{\partial z}\left(  \Delta_{m}\left(  G\left(  z\right)
		\right)  \right)  =\left(  r^{\overline{m-1}}{\displaystyle\prod
			\limits_{k=1}^{m-1}}r^{\overline{k-1}}k!\right)  e^{mz}\\
		\left(  \frac{m\left(  1-z\right)  ^{mr+\left(  m-1\right)  m}+\left(
			mr+m\left(  m-1\right)  \right)  \left(  1-z\right)  ^{mr+m\left(  m-1\right)
				-1}}{\left(  1-z\right)  ^{2(mr+m\left(  m-1\right)  )}}\right)  \\
		=\left(  r^{\overline{m-1}}{\displaystyle\prod\limits_{k=1}^{m-1}}%
		r^{\overline{k-1}}k!\right)  e^{mz}\times\alpha_{m}\left(  z\right)  ,
	\end{multline*}
	with
	\begin{align*}
		\alpha_{m}\left(  z\right)   &  =\frac{m}{\left(  1-z\right)  ^{m\left(
				r+m-1\right)  }}+\frac{m\left(  r+m-1\right)  }{\left(  1-z\right)  ^{m\left(
				r+m-1\right)  +1}}.
	\end{align*}
	The 2nd derivative of $ \Delta_{m}\left(G\left(  z\right)  \right) $ is
	\begin{align*}
		\frac{\partial^{2}}{\partial
			z^{2}}\left(  \Delta_{m}\left(  G\left(  z\right)  \right)  \right)
		&  =\left( r^{\overline{m-1}}
		{\displaystyle\prod\limits_{k=1}^{m-1}}
		r^{\overline{k-1}}k!\right)  e^{mz}\left[  m\times\alpha
		_{m}\left(  z\right)  +\frac{\partial}{\partial
			z}\alpha_{m}\left(  z\right)  \right] \\
		&   =\left(  r^{\overline{m-1}}
		{\displaystyle\prod\limits_{k=1}^{m-1}}
		r^{\overline{k-1}}k!\right)  e^{mz}\times\beta_{m}\left(
		z\right),
	\end{align*}
	with
	\begin{align*}
		\beta_{m}\left(  z\right)   &  =\frac{m^{2}}{\left(  1-z\right)  ^{m\left(r+m-1\right)  }}+\frac{2m^{2}\left(  r+m-1\right)  }{\left(  1-z\right)^{m\left(  r+m-1\right)  +1}}\\
		&  +\frac{m\left(  r+m-1\right)  \left(  m\left(  r+m-1\right)  +1\right)
		}{\left(  1-z\right)  ^{m\left(  r+m-1\right)  +2}}.
	\end{align*}
	Therefore
	\begin{multline*}
		\Delta_{2}\left(  \Delta_{m}\left(  G\left(  z\right)  \right)  \right)
		=\left(  r^{\overline{m-1}}{\displaystyle\prod\limits_{k=1}^{m-1}}%
		r^{\overline{k-1}}k!\right)  ^{2}e^{2mz}\\
		\times\left\vert
		\begin{array}
			[c]{c}%
			\frac{1}{\left(  1-z\right)  ^{m\left(  r+m-1\right)  }}\\
			\frac{m}{\left(  1-z\right)  ^{m\left(  r+m-1\right)  }}+\frac{m\left(
				r+m-1\right)  }{\left(  1-z\right)  ^{m\left(  r+m-1\right)  +1}}%
		\end{array}
		\right.  \\
		\left.
		\begin{array}
			[c]{c}%
			\frac{m}{\left(  1-z\right)  ^{m\left(  r+m-1\right)  }}+\frac{m\left(
				r+m-1\right)  }{\left(  1-z\right)  ^{m\left(  r+m-1\right)  +1}}\\
			\frac{m^{2}}{\left(  1-z\right)  ^{m\left(  r+m-1\right)  }}+\frac
			{2m^{2}\left(  r+m-1\right)  }{\left(  1-z\right)  ^{m\left(  r+m-1\right)
					+1}}+\frac{m\left(  r+m-1\right)  \left(  m\left(  r+m-1\right)  +1\right)
			}{\left(  1-z\right)  ^{m\left(  r+m-1\right)  +2}}%
		\end{array}
		\right\vert.
	\end{multline*}
	Thus,%
	\begin{align*}
		\Delta_{2}\left(  \Delta_{m}\left(  G\left(  z\right)  \right)  \right)   &
		=\left(  r^{\overline{m-1}}{\displaystyle\prod\limits_{k=1}^{m-1}}%
		r^{\overline{k-1}}k!\right)  ^{2}e^{2mz}\frac{1}{\left(  1-z\right)
			^{2m\left(  r+m-1\right)  }}\\
		&  {\small \times\left\vert
			\begin{array}
				[c]{cc}%
				1 & m+\frac{m\left(  r+m-1\right)  }{1-z}\\
				m+\frac{m\left(  r+m-1\right)  }{1-z} & m^{2}+\frac{2m^{2}\left(
					r+m-1\right)  }{1-z}+\frac{m\left(  r+m-1\right)  \left(  m\left(
					r+m-1\right)  +1\right)  }{\left(  1-z\right)  ^{2}}%
			\end{array}
			\right\vert}  \\
		&  =\left(  r^{\overline{m-1}}{\displaystyle\prod\limits_{k=1}^{m-1}%
		}r^{\overline{k-1}}k!\right)  ^{2}e^{2mz}\frac{m\left(  r+m-1\right)
		}{\left(  1-z\right)  ^{2m\left(  r+m-1\right)  +2}}.
	\end{align*}
	According to Sylvester's formula, we find
	\begin{align*}
		\Delta_{m+1}\left(  G\left(  z\right)  \right)   &  =\frac{\Delta_{2}\left(
			\Delta_{m}\left(  G\left(  z\right)  \right)  \right)  }{\Delta_{m-1}\left(
			G\left(  z\right)  \right)  }\\
		&  =\frac{\left(  r^{\overline{m-1}}{\displaystyle\prod\limits_{k=1}^{m-1}%
			}r^{\overline{k-1}}k!\right)  ^{2}e^{2mz}\frac{m\left(  r+m-1\right)
			}{\left(  1-z\right)  ^{2m\left(  r+m-1\right)  +2}}}{\left(  \left(
			r\right)  ^{\overline{m-2}}{\displaystyle\prod\limits_{k=1}^{m-2}}%
			r^{\overline{k-1}}k!\right)  \frac{e^{(m-1)z}}{\left(  1-z\right)  ^{\left(
					r+m-2\right)  \left(  m-1\right)  }}}.
	\end{align*}
	Thus,%
	\begin{align*}
		\Delta_{m+1}\left(  G\left(  z\right)  \right)   &  =\left(  r^{\overline{m-1}}{\displaystyle\prod\limits_{k=1}^{m-1}%
		}r^{\overline{k-1}}k!\right)  r^{\overline{m-1}}m!e^{mz+z}\frac{\left(
			r+m-1\right)  }{\left(  1-z\right)  ^{mr+m^{2}+r+2m}}\\
		&  =\left(  r^{\overline{m}}{\displaystyle\prod\limits_{k=1}^{m}}%
		r^{\overline{k-1}}k!\right)  e^{\left(  m+1\right)  z}\frac{1}{\left(
			1-z\right)  ^{\left(  m+1\right)  r+\left(  m+1\right)  m}}\\
		&  =\left(  \frac{e^{z}}{\left(  1-z\right)  ^{r}}\right)  ^{m+1}%
		\frac{r^{\overline{m}}{\displaystyle\prod\limits_{k=1}^{m}}r^{\overline{k-1}%
			}k!}{\left(  z-1\right)  ^{\left(  m+1\right)  m}}.
	\end{align*}
\end{proof}
\begin{lemma}\label{L13}
	For $n\geq1$ and $ z \neq 1$, we have
	\begin{multline*}
		\det\left(  \frac{e^{z}}{\left(  1-z\right)  ^{r}}\mathbb{D}_{i+j-2}%
		^{(r)}\left(  \frac{1}{1-z}\right)  \right)  _{1\leq i,j\leq n+1}=\\
		\left(  \frac{e^{z}}{\left(  1-z\right)  ^{r}}\right)  ^{n+1}\det\left(
		\mathbb{D}_{i+j-2}^{(r)}\left(  \frac{1}{1-z}\right)  \right)  _{1\leq i,j\leq
			n+1}.
	\end{multline*}
\end{lemma}
\begin{theorem}\label{T14}
	The Hankel determinant of order $n + 1$ of the generalized derangement polynomials $\mathbb{D}_{n}^{(r)}\left(  z\right)$ of order $r$, is given by
	\[
	\det\left(  \mathbb{D}_{i+j-2}^{(r)}\left(  z\right)  \right)  _{1\leq i,j\leq
		n+1}=z^{n\left(  n+1\right)  }r^{\overline{n}}
	{\displaystyle\prod\limits_{k=1}^{n}}
	r^{\overline{k-1}}k!.
	\]
\end{theorem}
\begin{proof}
	By Lemma \ref{L13}, we have
	\begin{align*}
		\left(  \frac{e^{z}}{\left(  1-z\right)  ^{r}}\right)  ^{n+1}\det\left(
		\mathbb{D}_{i+j-2}^{(r)}\left(  \frac{1}{1-z}\right)  \right)  &  =\det\left(  \frac{\partial^{i+j-2}}{\partial z^{i+j-2}}\left(
		\frac{e^{z}}{\left(  1-z\right)  ^{r}}\right)  \right)  \\
		&  =\left(  \frac{e^{z}}{\left(  1-z\right)  ^{r}}\right)  ^{n+1}\frac{
			r^{\overline{n}}
			{\displaystyle\prod\limits_{k=1}^{n}}
			r^{\overline{k-1}}k!}{\left(  z-1\right)  ^{\left(
				n+1\right)  n}},
	\end{align*}
	from where
	\begin{align*}
		\det\left(  \mathbb{D}_{i+j-2}^{(r)}\left(  \frac{1}{1-z}\right)  \right)  &  =\frac{r^{\overline{n}}
			{\displaystyle\prod\limits_{k=1}^{n}}
			r^{\overline{k-1}}k!}{\left(  z-1\right)  ^{\left(
				n+1\right)  n}}\\
		&  =\left(  -1\right)  ^{n(n+1)}\left(  \frac{1}{1-z}\right)  ^{^{\left(n+1\right)  n}}r^{\overline{n}}
		{\displaystyle\prod\limits_{k=1}^{n}}
		r^{\overline{k-1}}k!\\
		&  =\left(  \frac{1}{1-z}\right)  ^{^{\left(  n+1\right)  n}}r
		^{\overline{n}}
		{\displaystyle\prod\limits_{k=1}^{n}}
		r^{\overline{k-1}}k!.
	\end{align*}
	Now let's pose $T=\frac{1}{1-z}$, we obtain
	\[
	\det\left(  \mathbb{D}_{i+j-2}^{(r)}\left(  T\right)  \right)=T^{n\left(  n+1\right)}r^{\overline{n}}
	{\displaystyle\prod\limits_{k=1}^{n}}
	r^{\overline{k-1}}k!.
	\]	
\end{proof}
\begin{remark}
	When $r=1$ and $x=-1$, the well-known Hankel determinants of order $n+1$ of the matrices
	$[(i+j)!]$ and $[D_{i+j}]$, is given by
	\[
	det((i+j)!)_{0\leq i,j\leq n}=det(D_{i+j})_{0\leq i,j\leq n}=\left(
	{\displaystyle\prod\limits_{k=1}^{n}}
	k!\right)  ^{2}.
	\]
\end{remark}
\begin{theorem}
	The Hankel determinant of order $n + 1$ of the derangement polynomials $d_{n}^{(r)}\left(  z\right)$ of order $r$, is given by
	\[
	\det\left(d_{i+j-2}^{(r)}\left(  z\right) \right)  _{1\leq i,j\leq n+1}=r^{\overline{n}}
	{\displaystyle\prod\limits_{k=1}^{n}}
	r^{\overline{k-1}}k!.
	\]
\end{theorem}
\begin{proof} We have,
	\begin{align*}
		\det\left(d_{i+j-2}^{(r)}\left(  z\right) \right)  & =\left\vert
		\begin{array}
			[c]{cccc}%
			z^{0}\mathbb{D}_{0}^{(r)}\left(  \frac{1}{z}\right)   & z^{1}\mathbb{D}%
			_{1}^{(r)}\left(  \frac{1}{z}\right)   & \ldots & z^{n}\mathbb{D}_{n}%
			^{(r)}\left(  \frac{1}{z}\right)  \\
			z^{1}\mathbb{D}_{1}^{(r)}\left(  \frac{1}{z}\right)   & z^{2}\mathbb{D}%
			_{2}^{(r)}\left(  \frac{1}{z}\right)   & \ldots & z^{n+1}\mathbb{D}%
			_{n+1}^{(r)}\left(  \frac{1}{z}\right)  \\
			\ldotp & \ldotp & \ldotp & \ldotp\\
			z^{n}\mathbb{D}_{n}^{(r)}\left(  \frac{1}{z}\right)   & z^{n+1}\mathbb{D}%
			_{n+1}^{(r)}\left(  \frac{1}{z}\right)   & \ldots & z^{2n}\mathbb{D}%
			_{2n}^{(r)}\left(  \frac{1}{z}\right)
		\end{array}
		\right\vert \\
		& =\left(  z^{0}z^{1}z^{2}\cdots z^{n}\right)  ^{2}\left\vert
		\begin{array}
			[c]{cccc}%
			\mathbb{D}_{0}^{(r)}\left(  \frac{1}{z}\right)   & \mathbb{D}_{1}^{(r)}\left(
			\frac{1}{z}\right)   & \ldots & \mathbb{D}_{n}^{(r)}\left(  \frac{1}%
			{z}\right)  \\
			\mathbb{D}_{1}^{(r)}\left(  \frac{1}{z}\right)   & \mathbb{D}_{2}^{(r)}\left(
			\frac{1}{z}\right)   & \ldots & \mathbb{D}_{n+1}^{(r)}\left(  \frac{1}%
			{z}\right)  \\
			\ldotp & \ldotp & \ldotp & \ldotp\\
			\mathbb{D}_{n}^{(r)}\left(  \frac{1}{z}\right)   & \mathbb{D}_{n+1}%
			^{(r)}\left(  \frac{1}{z}\right)   & \ldots & \mathbb{D}_{2n}^{(r)}\left(
			\frac{1}{z}\right)
		\end{array}
		\right\vert \\
		& =\left(  z^{0}z^{1}z^{2}\cdots z^{n}\right)  ^{2}\left(  \frac{1}{z}\right)
		^{n\left(  n+1\right)  }r^{\overline{n}}{\displaystyle\prod\limits_{k=1}^{n}%
		}r^{\overline{k-1}}k!\\
		& =z^{n\left(  n+1\right)  -n\left(  n+1\right)  }r^{\overline{n}%
		}{\displaystyle\prod\limits_{k=1}^{n}}r^{\overline{k-1}}k!\\
		& =r^{\overline{n}}{\displaystyle\prod\limits_{k=1}^{n}}r^{\overline{k-1}}k!.
	\end{align*}
\end{proof}
\begin{corollary}
	The Hankel determinant of order $n + 1$ of the number of cyclic derangements $d_{n,r} $, is given by
	\[
	\det\left( d_{i+j-2,r} \right)  _{1\leq i,j\leq
		n+1}=r^{n\left(  n+1\right)}\left({\displaystyle\prod\limits_{k=1}^{n}}k!\right)  ^{2}.
	\]
\end{corollary}
\begin{proof}
	The relationship between the generalized derangement polynomials $\left(\mathbb{D}_{n}^{(r)}\left(  z\right)\right)$ of order $r$ and the number of cyclic derangements $d_{n,r}$ is given by
	\begin{equation}\label{d_{n,r}}
		(-1)^{n}\mathbb{D}_{n}^{(1)}\left(  -r\right)  =d_{n,r}.
	\end{equation}
	Now, using the Theorem \ref{T14} and the formula (\ref{d_{n,r}}), we get the desired result.
\end{proof}

\end{document}